\documentclass[12pt,a4paper]{amsart}
\usepackage{url}
\usepackage{amsmath,amsthm,amsfonts,amssymb,latexsym}

\newtheorem{theorem}{Theorem}[section]

\newtheorem{claim}[theorem]{Claim}

\theoremstyle{definition}

\newcommand{\IR}{\mathbb R}

\newcommand{\U}{\mathcal U}
\newcommand{\w}{\omega}

\newcommand{\A}{\mathcal{A}}
\newcommand{\D}{\mathcal{D}}
\newcommand{\CG}{\mathcal{G}}

\newcommand{\F}{\mathcal{F}}

\newcommand{\bb}{\mathfrak b}
\newcommand{\dd}{\mathfrak d}

\newcommand{\la}{\langle}
\newcommand{\ra}{\rangle}

\newcommand{\moi}{[\w]^{(\w,\w)}}

\newcommand{\nothing}[1]{}

\title[On  $M$-separability of topological spaces]{On  $M$-separability of countable spaces and
function spaces}

\author{Du\v{s}an Repov\v{s} and  Lyubomyr Zdomskyy}

\address{Faculty of Mathematics and Physics, and Faculty of Education,
University of Ljubljana, P. O. Box 2964, Ljubljana, Slovenia 1001.}
\email{dusan.repovs@guest.arnes.si}
\urladdr{http://www.fmf.uni-lj.si/\~{}repovs/index.htm}

\address{Kurt G\"odel Research Center for Mathematical Logic,
University of Vienna, W\"ahringer Stra\ss e 25, A-1090 Wien,
Austria.} \email{lzdomsky@gmail.com}
\urladdr{http://www.logic.univie.ac.at/\~{}lzdomsky/}

\thanks{
This research was supported in part by the  Slovenian Research Agency
 grants P1-0292-0101, J1-2057-0101 and BI-UA/09-10-005. The
second author would like to acknowledge the support of the FWF grant
P19898-N18.}

\begin{document}
\maketitle

\begin{abstract}
We study  $M$-separability as well as  some other combinatorial
versions of separability. In particular, we show that the
set-theoretic hypothesis $\mathfrak b=\mathfrak d$ implies that the
class of selectively separable spaces is not closed under finite
products, even for the spaces of continuous functions with the
topology of pointwise  convergence. We also show that there exists
no maximal $M$-separable countable space in the model of
Frankiewicz, Shelah, and Zbierski in which all  closed $P$-subspaces
of $\w^\ast$ admit an uncountable family of nonempty open mutually
disjoint subsets. This answers several questions of Bella,
Bonanzinga, Matveev, and Tkachuk.
\end{abstract}

\section{Introduction}

  Scheepers \cite{Sch99} introduced a number of combinatorial
  properties of a
topological space  stronger than  separability. In this paper we
concentrate mainly on  \emph{M-separability}\footnote{We follow the
terminology of \cite{BelBonMat09}.} defined as follows: a
topological space $X$ is said to be $M$-separable if for every
sequence $\la D_n:n\in\w\ra$ of dense subsets of $X$, one can pick
finite subsets $F_n\subset D_n$ such that $\bigcup_{n\in\w}F_n$ is
dense. A topological space $X$ is said to be \emph{maximal} if it has no isolated points
but any strictly stronger topology on $X$ has an isolated point.
The following theorems are the main results of this paper.

\begin{theorem} \label{main_max}
 It is consistent that no countable maximal space $X$ is  $M$-separable.
\end{theorem}

\begin{theorem}\label{main_front}
$(\bb=\dd).$ \ There exist subspaces $X_0$ and $X_1$ of $2^\w$ such
that $C_p(X_0)$ and $C_p(X_1)$ are $M$-separable, whereas
$C_p(X_0)\times C_p(X_1)$ is not.
\end{theorem}

Theorem \ref{main_max} answers \cite[Problem~3.3]{BelBonMatTka08} in the affirmative
and Theorem~\ref{main_front} shows that the negative answer to
\cite[Problems~3.7 and 3.9]{BelBonMatTka08} is consistent.

Regarding Theorem~\ref{main_max}, we show in section~\ref{new}  that  a
countable maximal space which is $M$-separable yields a separable closed
$P$-subset of $\w^\ast$, the remainder of the Stone-Czech compactification of $\w$.
A model of ZFC without c.c.c. (in particular separable)
closed $P$-subset of $\w^\ast$ was constructed in
\cite{FraSheZbi93}.
We recall that a subset $A$ of a topological space $X$ is called a $P$-subset,
if for every countable collection $\U$ of open neighborhoods of $A$
there exists an open neighborhood $V$ of $A$ such that $V\subset U$ for all $U\in\U$.

The proof of Theorem~\ref{main_front}
relies on the fact that for a metrizable separable space $X$,
$C_p(X)$ is $M$-separable if and only if all finite powers of $X$
have the Menger property (see \cite[\S~3]{BelBonMat09} and
references therein). We recall that a space $X$ is said to have
 the
\emph{Menger property} if for every sequence $\la u_n:n\in\w\ra$ of
open covers of $X$ there exists a sequence $\la v_n:n\in\w\ra$ such
that $v_n\in [u_n]^{<\w}$ and $\cup_{n\in\w}v_n$ is a cover of $X$.
Assuming $\bb=\dd$,  we construct in Section~\ref{construction}
 spaces $X_0,X_1\subset 2^\w$  all
of whose finite powers have the Menger property, whereas $X_0\times
X_1$ does not. Then the square of the disjoint union $X_0\sqcup X_1$
 does not have the   Menger property
 (since it contains a closed copy of $X_0\times X_1$, and the
Menger property is inherited by closed subspaces), and hence
$C_p(X_0\sqcup X_1)=C_p(X_0)\times C_p(X_1)$ fails to be
$M$-separable. At this point we would like to note that  it is not even known
whether there is a ZFC example of two spaces with the Menger property
whose product fails to have this property (see \cite[Problem~6.7]{Tsa07}).

Under CH Theorem~\ref{main_front} can be substantially improved.
Namely, by \cite[Theorem~2.1]{Bab09} there are spaces
$X,Y\subset\w^\w$  all finite powers of which have the Rothberger
property  whereas $X\times Y$ does not have the Menger property,
provided that CH holds. We recall that a space $X$ is said to have
 the \emph{Rothberger property} if for every sequence $\la u_n:n\in\w\ra$
of open covers of $X$ there exists a sequence $\la U_n:n\in\w\ra$
such that $U_n\in u_n$ and $\cup_{n\in\w}U_n=X$.

While preparing this manuscript we have learned from A.~Miller and
B.~Tsaban that
 CH implies the existence of
 $\gamma$-sets $Y_0, Y_1\subset
2^\w$ such that $Y_0\times Y_1$ does not have the Menger property.
It is known (see \cite{Tsa07} and references therein) that finite
powers of $\gamma$-sets are again $\gamma$-sets, and every
$\gamma$-set has the Rothberger property. On the other hand, Luzin
sets have the Rothberger property but they are not  $\gamma$-sets. Thus
this is an improvement of the  result of Babinkostova \cite{Bab09}
mentioned above.

Presently it is unknown whether the above-mentioned   construction
of $\gamma$-sets can be carried out under, e.g., $\w_1=\dd$.
Regarding the Babinkostova result, in the Laver model we have that
all sets with the Rothberger property are countable while $\mathfrak
b=\mathfrak d=\mathfrak c$. Therefore
 we still believe that
Theorem~\ref{main_front} can be of some interest.
\smallskip

In Section~\ref{citation} we provide  answers to a number of other
questions regarding various notions of separability. These are given
by citing results obtained in the framework of
 \emph{selection principles  in topology},
 a rapidly growing area of general topology (see e.g.,  \cite{Tsa07}).
In this way we hope to bring more attention to this area.

In what follows, by a space we understand a metrizable separable
topological space.

\section{Proof of Theorem~\ref{main_max}} \label{new}

Throughout the paper we standardly denote by
\begin{itemize}
\item $\w^\w$ the space of all functions from $\w$ to
$\w$ endowed with the Tychonov topology
(here $\w$ is equipped with the discrete topology);
\item $[\w]^\w$ the set of all infinite subsets of $\w$;
\item $[\w]^{<\w}$ the set of all finite subsets of $\w$; and
\item  $[\w]^{(\w,\w)}$ the set $\{a\subset\w:|a|=|\w\setminus
a|=\w\}$ of all infinite subsets of $\w$ with infinite complements.
\end{itemize}

A nonempty subset $\A\subset [\w]^\w$ is called a \emph{semifilter} \cite{BanZdo_sem}, if
for every $A\in\A$ and $X\subset\w$ such that $A\subset^\ast X$, $X\in\A$
($A\subset^\ast X$ means $|A\setminus X|<\w$).
A semifilter $\A$  is called a \emph{(free) filter}, if it is
closed under  finite intersections of its elements.
Filters which are maximal with respect to the inclusion are called \emph{ultrafilters}.
We recall that a filter $\A$ is a called a \emph{$P$-filter}, if
for every sequence $\la A_n:n\in\w\ra$ of elements of $\A$ there exists $A\in\A$
such that $A\subset^\ast A_n$ for all $n\in\w$.

For a semifilter  $\A\subset [\w]^\w$ we denote by $\A^\perp$ the set
$\{B\in [\w]^\w:\forall A \in \A\: (|A\cap B|=\w)\}$.
\medskip

Now suppose that  $(\w,\tau)$ is a countable maximal $M$-separable
space. We shall construct a separable $P$-subset of $\w^\ast$. This
suffices to prove Theorem~\ref{main_max} by the discussion following
it.

\begin{claim} \label{clcl1}
Every dense subset  $D$ of $\w$ is open, i.e. it belongs to $\tau$.
\end{claim}
\begin{proof}
Since $D$ is dense, the topology on $\w$ generated by $\tau\cup\{D\}$
has no isolated points.
If $D$ is not open, then  this topology  is strictly stronger than $\tau$.
\end{proof}

\begin{claim}\label{clcl1_2}
 Suppose that $\F$ and $\A$ are filters such that $\A\subset\F^\perp$.
Then there exists an ultrafilter $\U$ such that $\A\subset\U\subset\F^\perp$.
\end{claim}
\begin{proof}
Let $\U$ be a maximal filter with respect to the property $\A \subset\U \subset \F^\perp$.
We claim that $\U$ is an ultrafilter.
If this is not true, then there exists $X\subset\w$ such that $X,\w\setminus X\not\in\U$.
The maximality of $\U$ implies that
neither
 $\U\cup\{X\}$
nor $\U\cup\{\w\setminus X\}$ generates a filter contained in $\F^\perp$, which means that there
exist $U_0,U_1\in\U$ and $F_0,F_1\in\F$ such that $U_0\cap F_0\cap X=\emptyset$
and $U_1\cap F_1\cap (\w\setminus X)=\emptyset$. It follows that
$U_0\cap F_0\subset\w\setminus X$ and $U_1\cap F_1\subset X$, and hence
$(U_0\cap U_1)\cap (F_0\cap F_1)=\emptyset$, which contradicts the fact that
$\U\subset\F^\perp$.
\end{proof}

Let us denote by $\D$ the collection of all dense subsets of
$(\w,\tau)$. Claim~\ref{clcl1} implies that  $\D$ is a filter.
It is easy to verify that
  $(\bigcup_{n\in\w}\A_n)^\perp= \bigcap_{n\in\w}\A_n^\perp$
 for any semifilters
$\A,\A_0,\A_1,\ldots$  (see \cite{BanZdo_sem}).

\begin{claim} \label{clcl2}
There exists a sequence of ultrafilters $\la\U_n:n\in\w\ra$ such that $\D=\bigcap_{n\in\w}\U_n$.
\end{claim}
\begin{proof}
Let
$$\F_n=\{X\cup (A\setminus\{n\}):X\subset\w,\: n\in A\in\tau\}.$$
It is clear that $\F_n$ is a filter for every $n$ and
$\D=    (\bigcup_{n\in\w}\F_n)^\perp =\bigcap_{n\in\w}\F_n^\perp$.
 Claim~\ref{clcl1_2} yields for every $n$  an ultrafilter $\U_n$
such that $\D\subset\U_n\subset\F_n^\perp$.
It follows from the above  that
$$\D\subset\bigcap_{n\in\w}\U_n\subset \bigcap_{n\in\w}\F_n^\perp=\D,$$
 which completes the proof.
\end{proof}

The $M$-separability of $X$ simply means that
$\D$ is a $P$-filter. Thus we have proved that there exists a sequence
$\la\U_n:n\in\w\ra$ of ultrafilters such that $\bigcap_{n\in\w}\U_n$
is a $P$-filter. This obviously implies that the closure in $\w^\ast$
of $\{\U_n:n\in\w\}$ is a $P$-set, which finishes our proof.

\section{Proof of Theorem~\ref{main_front}} \label{construction}

First we introduce some notations and definitions.

The Cantor space $2^\w$ is identified with the power-set of $\w$ via
characteristic functions.  Each infinite subset $a$ of $\w$ can
also
be
 viewed as an element of $\w^\w$, namely the increasing
enumeration of $a$. Define a preorder $\leq^\ast$ on $\w^\w$ by
$f\leq^\ast g$ if and only if $f(n)\leq g(n)$ for all but finitely many
$n\in\w$. A subset $A\subset \w^\w$ is called
 \emph{dominating}
(resp. \emph{unbounded}), if for every $x\in\w^\w$ there exists
$a\in A$ such that $x\leq^\ast a$ (resp. $a\not\leq^\ast x$). The
minimal cardinality of an unbounded (resp. dominating) subset of
$\w^\w$ is denoted by $\bb$ (resp. $\dd$). It is a direct
consequence of the definition that $\bb\leq\dd$. The strict
inequality is consistent: it holds, e.g., in the Cohen model of
$\neg\mathrm{CH}$. For more information about $\bb,\dd$, and many
other cardinal characteristics of this kind we refer the reader to
\cite{Vau90}.

Given a relation $R$ on $\w$ and $x,y\in \w^\w$, we denote the set
$\{n\in\w : x(n)R\: y(n)\}$ by $[xR\:y]$.  For a filter $\F$ and elements
$x,y\in\w^\w$ we write $x\leq_\F y$ if $[x\leq y]\in \F$. The
relation $\leq_\F$ is easily seen to be a preoder. The minimal
cardinality of an unbounded with respect to $\leq_\F$ subset of
$\w^\w$ is denoted by $\bb(\F)$. It is easy to see that
$\bb\leq\bb(\F)\leq\dd$ for any filter $\F$,
$\leq^\ast=\leq_{\mathfrak Fr}$, and hence  $\bb=\bb(\mathfrak Fr)$,
where $\mathfrak Fr$ denotes the filter of all cofinite subsets of
$\w$.

For a filter $\F$, we say that $S=\{f_\alpha:\alpha<\bb(\F)\}$ is a
\emph{cofinal $\bb(\F)$-scale} if $f_\alpha\leq_\F f_\beta$ for all
$\alpha\leq\beta$, and for every $g\in\w^\w$ there exists
$\alpha<\bb(\F)$ such that $g\leq_\F f_\alpha$. Cofinal
$\bb(\mathfrak Fr)$-scales are simply called \emph{scales}.
 It is easy to see
that for every filter $\F$ there exists a cofinal $\bb(\F)$-scale
provided $\bb=\dd$.

The following fact is a direct consequence of
\cite[Theorem~4.5]{TsaZdo08}.

\begin{theorem} \label{cofscale1}
 Assume
that $\F$ is a filter and $S=\{f_\alpha:\alpha<\bb(\F)\}\subset
[\w]^\w$ is a cofinal $\bb(\F)$-scale. Then all finite powers of the
set $X=S\cup [\w]^{<\w}$ have the Menger property.
\end{theorem}

We shall also need the following characterization of the Menger
property which is due to Hurewicz (see \cite{Rec94}).

\begin{theorem} \label{dom_im}
Let $X$ be a zero-dimensional set of reals. Then $X$ has the Menger
property if and only if no continuous image of $X$ in $\w^\w$ is
dominating.
\end{theorem}

A family $\F\subset [\w]^\w$ is said to be \emph{centered} if each finite
subset of $\F$ has an infinite intersection. Centered families
generate filters by taking finite intersections and supersets.
We will denote the generated filter by $\langle \F \rangle$. For
$Y\subset\w^\w$, let $\mathrm{maxfin} Y$ denote its closure under
pointwise maxima of finite subsets. The proof of the following
theorem
 is reminiscent of that of Theorem~9.1 in \cite{TsaZdo08}.

\begin{theorem} \label{main}
$(\bb=\dd).$ \ There are subspaces $X_0$ and $X_1$ of $2^\w$ such
that all finite powers of $X_0$ and $X_1$ have the Menger property,
whereas $X_0\times X_1$ does not.
\end{theorem}
\begin{proof}
Let $\{d_\alpha : \alpha<\bb\}\subset \moi$ be a scale.

  Since $P:=[\w]^{(\w,\w)}\cup [\w]^{<\w}$ is
a  nowhere locally compact Polish space, it is
   homeomorphic
to $\mathbb Z^\w$. Therefore there exists a map $\star:P\times P\to
P$ which turns $P$ into a Polish topological group.

For $i\in 2$, we construct by induction on $\alpha<\bb$   a filter
$\F_i$  and a  dominating  $\bb(\F_i)$-scale $\{a^i_\alpha :
\alpha<\bb\}\subset \moi $   such that $a^0_\alpha\star
a^1_\alpha=\w\setminus d_\alpha$. Assume that $a^i_\beta$ have been
 defined for each $\beta<\alpha$ and $i\in 2$. Let $\A^i_\alpha  =
\mathrm{maxfin}\{d_\beta, a^i_\beta : \beta<\alpha\}$,
$\tilde{\F}^i_\alpha  = \bigcup_{\beta<\alpha}\F^i_\beta$, and
$\CG^i_\alpha  =  \{f\circ b : f\in\A^i_\alpha,
b\in\tilde{\F}^i_\alpha\}$, where $i\in 2$.

We inductively assume that $\F^i_\beta$, $\beta<\alpha$, is an
increasing chain of filters such that $|\F^i_\beta|\leq|\beta|$ for
each $\beta<\alpha$ and $i\in 2$. This implies that
$|\CG^i_\alpha|\leq|\alpha|<\bb$. Therefore there exists $c\in
[\w]^\w$ such that $x\leq^\ast c$  for all
$x\in\CG^0_\alpha\cup\CG^1_\alpha$. Since $Y_\alpha:=\{y\in \moi:
y\not\leq^\ast c\}$ is a dense $G_\delta$ subset of $\moi$, there
are $a^0_\alpha, a^1_\alpha\in Y_\alpha $ such that $a^0_\alpha\star
a^1_\alpha=\w\setminus d_\alpha$. Set
$$\F^i_\alpha = \langle \tilde{\F}^i_\alpha\cup\{[f\le a^i_\alpha] :
f\in\A^i_\alpha\}\rangle , \ i\in 2.$$
 We must show that
$\F^i_\alpha$'s remain  filters. Fix $i\in 2$. Since $\A^i_\alpha$
is closed under pointwise maxima, it suffices to show that
$b\cap[f\le a^i_\alpha]$ is infinite for all
$b\in\tilde{\F}^i_\alpha$ and $f\in\A^i_\alpha$. Suppose, to the
contrary,  that $b\cap[f\le a^i_\alpha]$ is finite. Then
$a^i_\alpha\leq a^i_\alpha\circ b \leq^\ast f\circ
b\in\CG^i_\alpha$, which contradicts with $a^i_\alpha\not\leq^\ast
c$ and $f\circ b\leq^\ast c$.

Set $X_i=\{a^i_\alpha : \alpha<\bb\}\cup [\w]^{<\w}$  and
$\F_i=\bigcup_{\alpha<\bb}\F^i_\alpha$, $i\in 2$. By
 construction,
$\{a^i_\alpha : \alpha<\bb\}$ is a cofinal $\bb(\F_i)$-scale. By
Theorem \ref{cofscale1}, all finite powers of $X_i$ have the Menger
property. Let $\phi:2^\w\to 2^\w$ be the map assigning to $x\subset
\w$ its complement $\w\setminus x$. It follows from the above that
$\{d_\alpha\}_{\alpha<\bb}\subset(\phi\circ\star)(X_0\times
X_1)\subset [\w]^\w$, and hence $X_0\times X_1$ can be continuously
mapped onto a dominating subset of $[\w]^\w$, which means that it
does not have the Menger property.
\end{proof}

One can also prove Theorem~\ref{main} by  methods developed in
\cite{ChaPol02} (see e.g., \cite{BanZdo??}). Moreover, one just has to ``add
an $\epsilon$'' to \cite{ChaPol02} to do this, and hence we believe that
Theorem~\ref{main} might be considered as a folklore for those who had a chance to
read \cite{ChaPol02}.

\section{Epilogue} \label{citation}

We recall from \cite{GerNag81} that   $X\subset 2^\w$ is called a
\emph{$\gamma$-set}, if $C_p(X)$ has the Fr\'echet-Urysohn property,
i.e. for every $f\in C_p(X)$ and a subset $A\subset C_p(X)$
containing $f$ in its closure, there exists a sequence of elements
of $A$ converging to $f$. The recent groundbreaking  result of
 Orenstein and Tsaban \cite{OreTsa??} states that under $\mathfrak
p=\mathfrak b$ there exists a $\gamma$-set of size $\bb$. Suppose
that $\mathfrak p=\mathfrak d$,  fix a $\gamma$-set
$X=\{x_\alpha:\alpha<\dd\}\subset 2^\w$ with $x_\alpha$'s mutually
different,  and a scale $S=\{f_\alpha:\alpha<\dd\}\subset\w^\w$.
Modify $S$ in such a way that it remains a scale and $\{n:
f_\alpha(n)$ is even$\}=x_\alpha$. We denote the modified scale
again by $S$. Then the $\gamma$-set $X$ is a continuous bijective
image of $S$, and hence $C_p(X)$ can be embedded into $C_p(S)$ as a
dense subset. Thus $C_p(S)$, which fails to be $M$-separable,
contains a dense subset which is $GN$-separable by
\cite[Theorems~86, 57, 40]{BelBonMat09} and the well known fact that
all finite powers of a $\gamma$-set have the Hurewicz as well as the
Rothberger properties (see \cite{BelBonMat09} for all the
definitions involved). Moreover, $C_p(X)$ is a dense subspace
of $\IR^{\dd}$, and $\{f\in C_p(X): f(X)\subset 2\}$ is a dense
subspace of $2^{\dd}$ which is $GN$-separable by
\cite[Proposition~90]{BelBonMat09}.
 This implies a positive answer to
 \cite[Questions~64, 93, and 94]{BelBonMat09} under $\mathfrak p=\dd$.
\smallskip

By \cite[Theorem~5.1]{JusMilSchSze96}, there exist a ZFC example of
a space $X\subset 2^\w$ of size $\w_1$ all of whose finite powers
have the  Hurewicz property. (Moreover, the space constructed in
Case 2 of the proof of \cite[Theorem~5.1]{JusMilSchSze96} is a
$\gamma$-set by  results of \cite{OreTsa??}.) Then $\{f\in C_p(X):
f(X)\subset 2\}$ is a dense hereditarily $H$-separable subspace of
$2^{\w_1}$ (see \cite[Theorem~40, Corollary~42]{BelBonMat09}). This
provides the positive answer to \cite[Problem~3.1]{BelBonMatTka08}.
\medskip

\noindent\textbf{Acknowledgments}. The authors would like to thank
Taras Banakh and Boaz Tsaban for many fruitful discussions regarding
properties of products of Menger spaces. We are particularly
grateful to Alan Dow  and the anonymous referee for bringing our
attention to \cite{FraSheZbi93} and \cite{Bab09}, respectively.

\end{document}